\documentclass{amsart}

\newtheorem{theorem}{Theorem}[section]
\newtheorem{lemma}[theorem]{Lemma}
\newtheorem{corollary}[theorem]{Corollary}

\theoremstyle{definition}

\theoremstyle{remark}
\newtheorem{remark}[theorem]{Remark}

\numberwithin{equation}{section}

\usepackage{hyperref}

\begin{document}

\title[DSM for H\"{o}lder continuous monotone operators]{Dynamical Systems Method for solving nonlinear equations
with locally H\"{o}lder continuous monotone operators}

\author{N. S. Hoang$\dag$}
\address{Mathematics Department, Kansas State University,
Manhattan, KS 66506-2602, USA}
\email{nguyenhs@math.ksu.edu}
\thanks{The author thanks Professor A. G. Ramm for useful comments which improve the paper.}



\date{}



\begin{abstract}
A version of the Dynamical Systems Method for solving ill-posed
nonlinear equations with monotone and locally H\"{o}lder continuous operators is studied in this paper.
A discrepancy principle is proposed and justified under natural and weak assumptions. 
The only smoothness assumption on $F$ is the local H\"{o}lder continuity of order $\alpha>1/2$. 
\end{abstract}

\maketitle

\section{Introduction}

In this paper we study a version of the Dynamical Systems Method (DSM) 
for solving the equation
\begin{equation}
\label{aeq1}
F(u)=f,
\end{equation}
 where
$F$ is a nonlinear monotone operator in a 
real Hilbert space $H$, and equation \eqref{aeq1} is assumed solvable, 
possibly nonuniquely. 
An operator $F$ is called monotone if 
\begin{equation}
\label{ceq2}
\langle F(u)-F(v),u-v\rangle\ge 0,\quad \forall u,v\in H.
\end{equation}
Here, $\langle \cdot,\cdot\rangle$ denotes the inner product in $H$.
It is known (see, e.g., \cite{R499}), that the set 
$\mathcal{N}:=\{u:F(u)=f\}$ is closed and convex if 
$F$ is monotone and continuous. 
A closed and convex set in a Hilbert space has a unique 
minimal-norm element. This element in 
$\mathcal{N}$ we denote by $y$, $F(y)=f$, and call it the minimal-norm 
solution to equation (1). 
We assume in addition that 
$F$ is locally H\"{o}lder continuous of order $\alpha>1/2$, i.e.,
\begin{equation}
\label{eqholder}
\|F(u) - F(v)\| \le C_R\|u-v\|^\alpha,\qquad \forall u,v\in B(y,R).
\end{equation}
Assume that $f=F(y)$ is not known but
 $f_\delta$, the noisy data, are known, and $\|f_\delta-f\|\le \delta$. If $F'(u)$
 is not boundedly invertible then 
 solving equation \eqref{aeq1} for $u$ given noisy data $f_\delta$ is 
often (but not always) an ill-posed problem.
When $F$ is a linear bounded operator many methods for stable solution 
of \eqref{aeq1} were proposed (see \cite{I}--\cite{R499} and the references therein). 
When $F$ is nonlinear several methods have been proposed and studied 
(see, e.g.,  \cite{Hanke1}, \cite{Kal1}, \cite{KNS}, \cite{JT}, \cite{OR}, \cite{Tautenhahn}, \cite{Tautenhahn2} 
and references therein). 
The most frequently used and studied methods are regularized Newton-type and gradient-type methods. 
These methods requires the knowledge of the Fr\'{e}chet derivative of $F$. 
Therefore, they are not applicable  
 if $F$ is not Fr\'{e}chet differentiable. Our goal in this paper is to 
study a method for a stable solution to problem \eqref{aeq1} when $F$ is not  Fr\'{e}chet differentiable. 

In this paper we study a version of the Dynamical Systems Method (DSM) for solving \eqref{aeq1}. 
In the formulation given in \cite{R499},   
the DSM consists of finding a nonlinear map $\Phi(t,u)$ such that the Cauchy problem
$$
\dot{u}=\Phi(t,u),\qquad u(0)=u_0,
$$
has a unique solution for all $t\ge0$, there exists $\lim_{t\to\infty}u(t):=u(\infty)$,
and $F(u(\infty))=f$,
\begin{equation}
\label{23eq1}
\exists !\, \,u(t)\quad \forall t\ge 0;\qquad \exists u(\infty);\qquad F(u(\infty))=f.
\end{equation}
Various choices of $\Phi$ satisfying \eqref{23eq1} were proposed and justified in \cite{R499}.
Each such choice yields a version of the DSM.

The DSM for solving equation \eqref{aeq1} was extensively studied in 
\cite{R499}--\cite{469}. 
In \cite{R499}, the following version of the DSM was investigated
for monotone operators $F$:
\begin{equation}
\label{hichehic}
\dot{u}_\delta = -\big{(}F'(u_\delta) + 
a(t)I\big{)}^{-1}\big{(}F(u_\delta)+a(t)u_\delta - f_\delta\big{)},
\quad u_\delta(0)=u_0.
\end{equation}
The convergence of this method was justified with an {\it a apriori} 
choice of stopping rule in \cite{R499}. 
An {\it a posteriori} choice of stopping rule for this method was proposed and justified in \cite{R574}. 
Another version of the DSM with an {\it a posteriori}
choice of stopping rule was formulated and 
justified in \cite{R549}. 

In this paper we consider the following version of the DSM for 
a stable solution to equation \eqref{aeq1}:
\begin{equation}
\label{aeq2}
\dot{u}_\delta = -\big{(}F(u_\delta)+a(t)u_\delta - f_\delta\big{)},
\quad u_\delta(0)=u_0,
\end{equation}
where $F$ is a monotone continuous operator and $u_0\in H$. 
It is known that a local solution of \eqref{aeq2} exists under the assumption that 
$F$ is monotone continuous and $a(t)>0$ (see, e.g., \cite[p.99]{D} and \cite[p.165]{R499}). 
When $\delta = 0$ and $a(t)$ satisfies some conditions then it is known that 
the solution to \eqref{aeq2} exists globally (see, e.g., \cite[p.170]{R499}). 

The advantage of the method in \eqref{aeq2} compared with
the one in \eqref{hichehic} is the absence of the inverse operator in the algorithm, which
makes the algorithm \eqref{aeq2} less expensive than \eqref{hichehic}.
On the other hand, algorithm  \eqref{hichehic} converges faster than  
\eqref{aeq2} in many cases. 
Another advantage of the DSM \eqref{aeq2} is the applicability when $F$ is locally H\"{o}lder continuous of order $\alpha>0$ 
but not Fr\'{e}chet 
differentiable as shown in this paper.  

The convergence of the method \eqref{aeq2} for any initial value $u_0$ 
with an {\it a priori} choice of stopping rule was justified in \cite[p.170]{R499}.
%
In \cite{R550} the DSM \eqref{aeq2} with a stopping rule of Discrepancy Principle type was proposed and 
justified under the assumption that $F$ is Fr\'{e}chet differentiable. 
%
%
There, convergence of $u_\delta(t_\delta)$, chosen by a stopping rule of Discrepancy Principle 
type, is proved for the regularizing function $a(t) = d/(c+t)^b$ where $c\ge 1$, $b\in(0,1/2)$ 
and $d$ is sufficiently large. However, how large one should choose the parameter $d$ is not quantified in \cite{R550}. 

In this paper we study the DSM \eqref{aeq2} with the stopping rule proposed in \cite{R550} 
under weaker assumption on $F$ and for a larger class of regularizing function $a(t)$. 
The novel results in this paper include a justification of  
the DSM \eqref{aeq2} with our stopping rule for a stable solution to \eqref{aeq1} 
 under the assumption that $F$ is locally H\"{o}lder continuous of order $\alpha>1/2$. 
This condition is much weaker than the Fr\"{e}chet differentiability of $F$ which was used 
in \cite{R550}. 
Moreover, our results are justified for a larger class of regularizing function $a(t)$. 
The main results of this paper are Theorem~\ref{theorem1} and Theorem~\ref{theorem2} in which
a DP is formulated,  
the existence of a stopping time $t_\delta$ is proved, and
the convergence of the DSM with the proposed DP is justified under some weak and 
natural assumptions.

\section{Auxiliary results}

Let us consider the following equation:
\begin{equation}
\label{2eq2}
F(V_\delta)+aV_\delta-f_\delta = 0,\qquad a>0,
\end{equation}
where $a=const$. It is known (see, e.g., \cite{R499}, \cite{Z}) that equation \eqref{2eq2} with monotone continuous operator $F$ has a unique solution for any $f_\delta\in H$. 

Let us recall the following result from \cite{R499}:
\begin{lemma}
\label{rejectedlem}
Assume that equation \eqref{aeq1} is solvable, $y$ is its minimal-norm 
solution, and $F$ is monotone continuous. Then
$$
\lim_{a\to 0} \|V_{a}-y\| = 0,
$$
where $V_a$ solves \eqref{2eq2} with $\delta=0$.
\end{lemma}

Let $a=a(t)$ be a strictly monotonically decaying continuous
positive function on $[0,\infty)$, $0<a(t)\searrow 0$, and assume $a\in C^1[0,\infty)$. 
These assumptions hold throughout the paper and often are not repeated. 
Then the solution $V_\delta$ of \eqref{2eq2} is a function of $t$, $V_\delta=V_\delta(t)$.

Below the words decreasing and increasing mean strictly decreasing and strictly increasing.

\begin{lemma}
\label{remark1}
Assume $\|F(0)-f_\delta\|>0$.
Let $0<a(t)\searrow 0$, and $F$ be monotone.
Denote 
$$
\psi(t) :=\|V_\delta(t)\|,\qquad \phi(t):=a(t)\psi(t)=\|F(V_\delta(t)) - f_\delta\|,
$$ 
where $V_\delta(t)$ solves \eqref{2eq2} with $a=a(t)$. 
Then
$\phi(t)$ is decreasing, and $\psi(t)$ is increasing.
\end{lemma}
 
\begin{proof}
Since $\|F(0)-f_\delta \|>0$, one has $\psi(t)\not=0,\, \forall t\ge 0$. 
Indeed, if $\psi(t)\big{|}_{t=\tau}=0$, then $V_\delta(\tau)=0$, and equation \eqref{2eq2} implies
$\|F(0)-f_\delta\|=0$, which is a contradiction. 
Note that $\phi(t)=a(t)\|V_\delta(t)\|$. One has
\begin{equation}
\label{1eq3}
\begin{split}
0&\le \langle F(V_\delta(t_1))-F(V_\delta(t_2)),V_\delta(t_1)-V_\delta(t_2)\rangle\\
&= \langle -a(t_1)V_\delta(t_1)+a(t_2)V_\delta(t_2),V_\delta(t_1)-V_\delta(t_2)\rangle\\
&= (a(t_1)+a(t_2))\langle V_\delta(t_1),V_\delta(t_2) \rangle -a(t_1)\|V_\delta(t_1)\|^2 - a(t_2)\|V_\delta(t_2)\|^2.
\end{split}
\end{equation}
Thus,
\begin{equation}
\label{2eq6}
\begin{split}
0&\le (a(t_1)+a(t_2))\langle V_\delta(t_1),V_\delta(t_2) \rangle -a(t_1)\|V_\delta(t_1)\|^2 - a(t_2)\|V_\delta(t_2)\|^2\\
& \le  (a(t_1)+a(t_2))\|V_\delta(t_1)\|\|V_\delta(t_2) \| - a(t_1)\|V_\delta(t_1)\|^2 - a(t_2)\|V_\delta(t_2)\|^2\\
& = (a(t_1) \|V_\delta(t_1)\| - a(t_2) \|V_\delta(t_2)\|)(\|V_\delta(t_2)\|-\|V_\delta(t_1)\|)\\
& = (\phi(t_1)-\phi(t_2))(\psi(t_2) - \psi(t_1)).
\end{split}
\end{equation}

If $\psi(t_2) > \psi(t_1)$, then \eqref{2eq6} implies $\phi(t_1)\ge \phi(t_2)$, so
$$
a(t_1)\psi(t_1)\ge a(t_2)\psi(t_2)> a(t_2)\psi(t_1).
$$
Thus, if $\psi(t_2)> \psi(t_1)$, then $a(t_2)< a(t_1)$ and, therefore, $t_2> t_1$,
because $a(t)$ is strictly decreasing. 

Similarly, if $\psi(t_2)< \psi(t_1)$, then $\phi(t_1)\le \phi(t_2)$. This implies $a(t_2)> a(t_1)$, so $t_2< t_1$.

Suppose $\psi(t_1)=\psi(t_2)$, i.e., $\|V_\delta(t_1)\|=\|V_\delta(t_2)\|$. From \eqref{1eq3}, one has
$$
\|V_\delta(t_1)\|^2\le \langle V_\delta(t_1),V_\delta(t_2) \rangle \le \|V_\delta(t_1)\|\|V_\delta(t_2)\| = \|V_\delta(t_1)\|^2.
$$
This implies $V_\delta(t_1)=V_\delta(t_2)$, and then equation \eqref{2eq2} implies $a(t_1)=a(t_2)$. 
Hence, $t_1=t_2$, because $a(t)$ is strictly decreasing.

Therefore, $\phi(t)$ is decreasing
and $\psi(t)$ is increasing.
\end{proof}

\begin{lemma}
\label{lem1}
Let $F$ be a monotone continuous operator. 
Then, 
\begin{equation}
\label{eqhnnxsa}
\lim_{t\to\infty}\|F(V_\delta(t))-f_\delta\|\le \delta.
\end{equation}
\end{lemma}

\begin{proof}
We have $F(y)=f$, and
\begin{align*}
0=&\langle F(V_\delta)+aV_\delta-f_\delta, F(V_\delta)-f_\delta \rangle\\
=&\|F(V_\delta)-f_\delta\|^2+a\langle V_\delta-y, F(V_\delta)-f_\delta \rangle + a\langle y, F(V_\delta)-f_\delta \rangle\\
=&\|F(V_\delta)-f_\delta\|^2+a\langle V_\delta-y, F(V_\delta)-F(y) \rangle + a\langle V_\delta-y, f-f_\delta \rangle \\
&+ a\langle y, F(V_\delta)-f_\delta \rangle\\
\ge&\|F(V_\delta)-f_\delta\|^2 + a\langle V_\delta-y, f-f_\delta \rangle + a\langle y, F(V_\delta)-f_\delta \rangle.
\end{align*}
Here the inequality $\langle V_\delta-y, F(V_\delta)-F(y) \rangle\ge0$ was used. 
Therefore,
\begin{equation}
\label{1eq1}
\begin{split}
\|F(V_\delta)-f_\delta\|^2 &\le -a\langle V_\delta-y, f-f_\delta \rangle - a\langle y, F(V_\delta)-f_\delta \rangle\\
&\le a\|V_\delta-y\| \|f-f_\delta\| + a\|y\| \|F(V_\delta)-f_\delta\|\\
&\le  a\delta \|V_\delta-y\|  + a\|y\| \|F(V_\delta)-f_\delta\|.
\end{split}
\end{equation}
On the other hand, we have
\begin{align*}
0&= \langle F(V_\delta)-F(y) + aV_\delta +f -f_\delta, V_\delta-y\rangle\\
&=\langle F(V_\delta)-F(y),V_\delta-y\rangle + a\| V_\delta-y\| ^2 + 
a\langle y, V_\delta-y\rangle + \langle f-f_\delta, V_\delta-y\rangle\\
&\ge  a\| V_\delta-y\| ^2 + a\langle y, V_\delta-y\rangle + \langle 
f-f_\delta, V_\delta-y\rangle,
\end{align*}
where the inequality 
$\langle V_\delta-y, F(V_\delta)-F(y) \rangle\ge0$ was used. Therefore,
$$
a\|V_\delta-y\|^2 \le a\|y\|\|V_\delta-y\|+\delta\|V_\delta-y\|.
$$
This implies
\begin{equation}
\label{1eq2}
a\|V_\delta-y\|\le a\|y\|+\delta.
\end{equation}
From \eqref{1eq1} and \eqref{1eq2}, and an elementary inequality $ab\le \epsilon a^2+\frac{b^2}{4\epsilon},\,\forall\epsilon>0$, one gets:
\begin{equation}
\label{3eq4}
\begin{split}
\|F(V_\delta)-f_\delta\|^2&\le \delta^2 + a\|y\|\delta + a\|y\| \|F(V_\delta)-f_\delta\|\\
&\le \delta^2 + a\|y\|\delta + \epsilon \|F(V_\delta)-f_\delta\|^2 + 
\frac{1}{4\epsilon}a^2\|y\|^2,
\end{split}
\end{equation}
where $\epsilon>0$ is fixed, independent of $t$, and can be chosen 
arbitrary small. 
Let $t\to\infty$ and $a=a(t)\searrow 0$. Then \eqref{3eq4} implies
$$
\lim_{t\to\infty}(1-\epsilon)\|F(V_\delta(t))-f_\delta\|^2\le 
\delta^2,
$$
for any fixed $\epsilon>0$ arbitrarily small. 
This implies \eqref{eqhnnxsa}. Lemma~\ref{lem1} is proved.
\end{proof}

\begin{remark}
\label{rem3}
{\rm 
Let $V:=V_\delta(t)|_{\delta=0}$, so 
$$
F(V)+a(t)V-f=0.
$$ 
Let $y$ be the minimal-norm solution to equation \eqref{aeq1}. 
We claim that
\begin{equation}
\label{rejected11}
\|V_{\delta}-V\|\le \frac{\delta}{a}.
\end{equation}
Indeed, from \eqref{2eq2} one gets
$$
F(V_{\delta}) - F(V) + a (V_{\delta}-V)=f_\delta - f.
$$
Multiply this equality with $(V_{\delta}-V)$ and use the monotonicity of $F$ to get
\begin{align*}
a \|V_{\delta}-V\|^2\le \delta \|V_{\delta}-V\|.
\end{align*}
This implies \eqref{rejected11}. 

Similarly, multiplying the equation
$$
F(V) + a V -F(y)=0,
$$
by $V-y$ one derives the inequality:
\begin{equation}
\label{rejected12}
\|V\| \le \|y\|.
\end{equation}
Similar arguments one can find in \cite{R499}. 

From \eqref{rejected11} and \eqref{rejected12}, one gets the following estimate:
\begin{equation}
\label{2eq1}
\|V_\delta\|\le \|V\|+\frac{\delta}{a}\le \|y\|+\frac{\delta}{a}.
\end{equation}

From the monotonicity of $F$ and \eqref{2eq2} one gets
\begin{equation}
\label{eq>1}
\begin{split}
0 &\le \langle F(V_\delta(t)) - F(V_\delta(t')), V_\delta(t) - V_\delta(t') \rangle \\
&\le \langle a(t')V_\delta(t') - a(t)V_\delta(t), V_\delta(t) - V_\delta(t') \rangle \\
& = - a(t')\|V_\delta(t) - V_\delta(t')\|^2 + 
(a(t') - a(t)) \langle V_\delta(t), V_\delta(t) - V_\delta(t') \rangle\\
& \le - a(t')\|V_\delta(t) - V_\delta(t')\|^2 + 
|a(t') - a(t)| \|V_\delta(t)\| \|V_\delta(t) - V_\delta(t')\|,
\end{split}
\end{equation}
for all $t,t'>0$. 
This implies:
\begin{equation}
\label{eq>2}
\limsup_{\xi\to 0} \frac{\|V_\delta(t+\xi) - V_\delta(t)\|}{|\xi|} 
\le \frac{|\dot{a}(t)|}{a(t)} \|V_\delta(t)\|,\qquad t>0.
\end{equation}
}
\end{remark}

Let us formulate and prove a version of the Gronwall's inequality for 
continuous functions.

\begin{lemma}
\label{lemma2.7}
Let $\alpha(t)$ and $\beta(t)$ be continuous nonnegative functions on $[0,\infty)$.  
Let $0\le g(t)$ be a continuous function on $[0,\infty)$ satisfying the following condition:
\begin{equation}
\label{eqw1}
\limsup_{\xi\to 0} \frac{g^2(t + \xi) -g^2(t)}{\xi} 
\le -2\alpha(t)g^2(t) + 2\beta(t)g(t),\qquad \forall t\ge 0.
\end{equation}
Then
\begin{equation}
\label{eqw5}
g(t) \le g(0)e^{-\tilde{\varphi}(t)} + e^{-\tilde{\varphi}(t)}\int_0^t e^{\tilde{\varphi}(s)}\beta(s)ds
,\qquad\tilde{\varphi}(t) := \int_0^t \alpha(s)ds.
\end{equation}
\end{lemma}

\begin{proof}
Let 
$$
g_\epsilon(t):=\bigg{(}g^2(t)+\epsilon e^{-2\int_0^t \alpha(\xi)d\xi}\bigg{)}^\frac{1}{2},\qquad t\ge 0,\qquad\epsilon>0.$$ 
From \eqref{eqw1} one obtains
\begin{equation}
\label{eqjj1}
\begin{split}
\limsup_{\xi\to 0} \frac{g_\epsilon^2(t + \xi) -g_\epsilon^2(t)}{\xi} 
&= \limsup_{\xi\to 0} \frac{g^2(t + \xi) -g^2(t)}{\xi} + \epsilon\frac{d}{dt} e^{-2\int_0^t \alpha(\xi)d\xi}\\
&\le -2\alpha(t)g^2(t) + 2\beta(t)g(t) - 2\epsilon \alpha(t)e^{-2\int_0^t \alpha(\xi)d\xi}\\
&\le -2 \alpha(t) g_\epsilon^2(t) + 2\beta(t)g_\epsilon(t),\qquad \forall t\ge 0.
\end{split}
\end{equation}
Since $g_\epsilon(t)>0$, it follows from \eqref{eqw1} and the continuity of $g_\epsilon$ that
\begin{equation}
\label{eqw2}
\limsup_{\xi\to 0} \frac{g_\epsilon(t + \xi) -g_\epsilon(t)}{\xi} 
\le -\alpha(t)g_\epsilon(t) + \beta(t).
\end{equation}
From the Taylor expansion of $e^{\int_t^{t+\xi}a(s)ds}$, we have
$$
e^{\tilde{\varphi}(t+\xi)} = e^{\tilde{\varphi}(t)+\int_t^{t+\xi}a(s)ds} 
= e^{\tilde{\varphi}(t)}\bigg{(}1 + \int_t^{t+\xi}a(s)ds + O(\xi^2)\bigg{)},\qquad \xi\to 0.
$$
This, \eqref{eqw2}, the mean value theorem for integration, and the continuity of $g_\epsilon(t)$ imply
\begin{equation}
\label{eqw3}
\limsup_{\xi\to 0} \frac{e^{\tilde{\varphi}(t+\xi)}g_\epsilon(t + \xi) -e^{\tilde{\varphi}(t)}g_\epsilon(t)}{\xi} 
\le e^{\tilde{\varphi}(t)}\beta(t).
\end{equation}
From \eqref{eqw3} one obtains 
\begin{equation}
\label{eqw4}
e^{\tilde{\varphi}(t)}g_\epsilon(t) - e^{\tilde{\varphi}(0)}g_\epsilon(0) \le \int_0^t e^{\tilde{\varphi}(s)}\beta(s)ds,\qquad t\ge 0.
\end{equation}
This implies 
\begin{equation}
\label{eqjj3}
g(t)<g_\epsilon(t) \le \bigg{(}g^2(0) + \epsilon \bigg{)}^\frac{1}{2} e^{-\tilde{\varphi}(t)} + e^{-\tilde{\varphi}(t)} 
\int_0^t e^{\tilde{\varphi}(s)}\beta(s)ds,\qquad \forall t\ge 0.
\end{equation}
Letting $\epsilon\to 0$ in \eqref{eqjj3} one obtains \eqref{eqw5}. Lemma~\ref{lemma2.7} is proved. 
\end{proof}

\begin{lemma}
\label{lemma2.8}
Let $a(t)\in C^{1}[0,\infty)$ satisfy the following conditions (see also \eqref{eqsxa})
\begin{equation}
\label{eqpp1}
0< a(t)\searrow 0,\qquad 0<\frac{|\dot{a}(t)|}{a^2(t)} \searrow 0.
\end{equation}
Let $\phi(t) := \int_0^t a(s)ds$ and $V_\delta(t)$ be the solution to \eqref{2eq2} 
with $a=a(t)$. Then the following relations hold:
\begin{align}
\label{eqsxaz}
\lim_{t\to\infty} \phi(t) &= \infty,\\
\label{eql6}
\lim_{t\to \infty} e^{r\phi(t)}a(t) &= \infty,\qquad r=const>0,\\
\label{eql3}
\lim_{t\to\infty} \frac{\int_0^t 
e^{\phi(s)}|\dot{a}(s)|\|V_\delta(s)\|ds}{e^{\phi(t)}}&=0,\\
\label{eql9}
M:=\lim_{t\to\infty} \frac{\int_0^{t} e^{\phi(s)}|\dot{a}(s)|ds}{e^{\phi(t)}a(t)} &= 0.
\end{align}
\end{lemma}

\begin{proof}
{\it Let us first prove \eqref{eqsxaz}.} 
It follows from \eqref{eqpp1} that there exists $t_1\ge 0$ such that 
$$
a(t) \ge -\frac{\dot{a}(t)}{a(t)},\qquad \forall t\ge t_1. 
$$
This implies
\begin{equation}
\label{eqsxazz}
\phi(t)\ge \int_{t_0}^t a(s)ds \ge \int_{t_0}^t \frac{-\dot{a}(s)}{a(s)}ds = -\ln a(s)\bigg{|}_{t_0}^t=\ln a(t_0) -\ln a(t).
\end{equation}
Relation \eqref{eqsxaz} follows from the relation $\lim_{t\to\infty}a(t)=0$ and \eqref{eqsxazz}.

{\it Let us prove \eqref{eql6}.} 
We claim that, for sufficiently large $t>0$, the following inequality holds:
\begin{equation}
\label{xex5}
\phi(t) = \int_0^t a(s)ds > \frac{1}{r}\ln \frac{1}{a^2(t)}.
\end{equation}
Indeed, by L'Hospital's rule and \eqref{eqpp1}, one gets
\begin{equation}
\label{xex6}
\lim_{t\to\infty} \frac{\phi(t)}{\ln \frac{1}{a^2(t)}} 
= \lim_{t\to\infty}\frac{a^2(t)}{2|\dot{a}(t)|} = \infty.
\end{equation}
This implies that \eqref{xex5} holds for all $t\ge \tilde{T}$ provided that $\tilde{T}>0$ is sufficiently large.
It follows from inequality \eqref{xex5} that 
\begin{equation}
\label{xex7}
\lim_{t\to\infty} a(t)e^{r\phi(t)} \ge \lim_{t\to\infty} a(t) e^{\ln \frac{1}{a^2(t)}} = \lim_{t\to\infty}\frac{1}{a(t)} = \infty.
\end{equation}
Thus, equality \eqref{eql6} is proved.

{\it Let us prove \eqref{eql3}.} 
Since $a(t)\|V_\delta(t)\|$ is a decreasing function of $t$ (cf. Lemma~\ref{remark1}) one gets
\begin{equation}
\label{eql4}
\lim_{t\to\infty} \frac{\int_0^t 
e^{\phi(s)}|\dot{a}(s)|\|V_\delta(s)\|ds}{e^{\phi(t)}}\le
\lim_{t\to\infty} \frac{\int_0^t 
e^{\phi(s)}\frac{|\dot{a}(s)|}{a(s)}a(0)\|V_\delta(0)\|ds}{e^{\phi(t)}}.
\end{equation}
We claim that
\begin{equation}
\label{eql5}
\lim_{t\to\infty} \frac{\int_0^t 
e^{\phi(s)}\frac{|\dot{a}(s)|}{a(s)}ds}{e^{\phi(t)}}=0.
\end{equation}
Indeed, if $\int_0^t 
e^{\phi(s)}\frac{|\dot{a}(s)|}{a(s)}ds<\infty$, then \eqref{eql5} follows from \eqref{eqsxaz}.
Otherwise, relation \eqref{eql5} follows from L'Hospital's rule and 
the relation $\lim_{t\to\infty}\frac{|\dot{a}(t)|}{a^2(t)}=0$. 

From \eqref{eql4} and \eqref{eql5} one gets \eqref{eql3}.

{\it Let us prove \eqref{eql9}}.  
Since \eqref{eqsxaz} 
 holds and $a(t)e^{\phi(t)}\to 0$ as $t\to\infty$, by \eqref{eql6} with $r=1$,
relation \eqref{eql9} holds if the numerator of \eqref{eql9} is bounded. 
Otherwise, L'Hospital's rule yields
\begin{equation}
\label{eql14}
M=\lim_{t\to\infty} \frac{ e^{\phi(t)}|\dot{a}(t)|}
{e^{\phi(t)}a^2(t) + e^{\phi(t)}\dot{a}(t)} = 0.
\end{equation}
Here we have used the relation $\lim_{t\to\infty}\frac{|\dot{a}|(t)}{a^2(t)}=0$. 

Lemma~\ref{lemma2.8} is proved. 
\end{proof}

\begin{remark}
From \eqref{eql3} and the inequality $\|V_\delta(t)\|\ge \|V_\delta(0)\|>0,\,\forall t\ge 0$, (see Lemma~\ref{remark1}), one gets 
the following relation
\begin{equation}
\label{eqwer}
\lim_{t\to\infty} \frac{\int_0^t 
e^{\phi(s)}|\dot{a}(s)|ds}{e^{\phi(t)}}=0.
\end{equation}

Let $\epsilon>0$ be arbitrary. It follows from \eqref{eql9} that there exists $t_\epsilon>0$ 
such that the following inequality holds:
\begin{equation}
\label{eqwer2}
e^{-\phi(t)}\int_0^{t} e^{\phi(s)}|\dot{a}(s)|ds < \epsilon a(t),\qquad \forall t\ge t_\epsilon. 
\end{equation}
\end{remark}

Let us assume that $a(t)$ satisfies the following conditions:
\begin{equation}
\label{eqsxaaa}
0<a(t)\searrow 0,\qquad \frac{1}{2}>q > \frac{|\dot{a}(t)|}{a^2(t)}\searrow 0.
\end{equation}

\begin{lemma}
\label{lemauxi2}
Let $a(t)$ 
satisfy \eqref{eqsxaaa} and $\varphi(t):=(1-q)\int_0^t a(s)ds$, $q\in(0,1/2)$. 
Then one has
\begin{equation}
\label{auxieq3}
e^{-\varphi(t)}\int_0^t e^{\varphi(s)}|\dot{a}(s)|\|V_\delta(s)\|ds \le
 \frac{q}{1-2q}a(t)\|V_\delta(t)\|,\qquad t\ge 0.
\end{equation} 
\end{lemma}

\begin{proof} 

Let us prove that
\begin{equation}
\label{eqv1}
e^{\varphi(t)} |\dot{a}(t)|\le \frac{q}{1-2q}\bigg{(}a(t)e^{\varphi(t)}\bigg{)}',\qquad \forall t\ge 0.
\end{equation}
Inequality \eqref{eqv1} is equivalent to
\begin{equation}
\label{eqv2}
\bigg{(}\frac{1}{q}-2\bigg{)} e^{\varphi(t)} |\dot{a}(t)| \le \dot{a}(t)e^{\varphi(t)}
+ (1-q)a^2(t)e^{\varphi(t)}
,\qquad t\ge 0.
\end{equation}
Note that $\dot{a} = -|\dot{a}|$. 
Inequality \eqref{eqv2} holds because from \eqref{eqsxaaa} one obtains
\begin{equation}
\label{eqv3}
\bigg{(}\frac{1}{q}-2\bigg{)} |\dot{a}(t)| < - |\dot{a}(t)| + (1-q)a^2(t),\qquad t\ge 0.
\end{equation}
Thus, inequality \eqref{eqv1} holds. Integrate \eqref{eqv1} from 0 to $t$ and get
\begin{equation}
\label{eqv4}
\int_0^t e^{\varphi(s)} |\dot{a}(s)| ds \le \frac{q}{1-2q} \bigg{(}a(t)e^{\varphi(t)} - a(0)e^{0}\bigg{)}
< \frac{q}{1-2q}a(t)e^{\varphi(t)},\quad \forall t\ge 0.
\end{equation}

Multiplying \eqref{eqv4} by $e^{-\varphi(t)}\|V_\delta(t)\|$ and using the fact that $\|V_\delta(t)\|$
is increasing, one gets inequality \eqref{auxieq3}. Lemma~\ref{lemauxi2} is proved.
\end{proof}

\section{Main results}
\label{mainsec}

\subsection{Dynamical Systems Method}

Let $u_\delta(t)$ solve the following Cauchy problem:
\begin{equation}
\label{3eq12}
\dot{u}_\delta = -[F(u_\delta)+a(t)u_\delta-f_\delta],\qquad u_\delta(0)=u_0.
\end{equation}
Assume
\begin{equation}
\label{eqsxa}
0<a(t)\searrow 0,\qquad 0< \frac{|\dot{a}|(t)}{a^2(t)}\searrow 0,\qquad t\ge 0.
\end{equation}

\begin{remark}{\rm
Let $a(t)=\frac{d}{(c+t)^b}$, 
where $b\in(0,1)$,\, 
$c>0$ and $d>0$. Then this $a(t)$ satisfies \eqref{eqsxa}. 
}
\end{remark}

\begin{remark}
\label{remark3.2}
{\rm
It is known that there exists a unique local solution to problem \eqref{3eq12} 
for any initial data $u_0$ 
if $F$ is monotone continuous and $0<a(t)$ is a continuous function. 
Proofs for this are often based on Peano approximations (see, e.g., \cite[p.99]{D} and \cite[p.165]{R499}). 
When $F$ is monotone and hemicontinuous then equation \eqref{3eq12} is understood in the 
weak sense. 
When $F$ is monotone and continuous it is known that equation \eqref{3eq12} can be understood in the 
strong sense (see, e.g., \cite[p.167]{R499}). 
}
\end{remark}

The following lemma guarantees the global existence of a unique solution to \eqref{3eq12}. 

\begin{lemma}
\label{lemma9.4}
Let $F$ be monotone and continuous. Let $0<a(t)$ be a continuous function satisfying \eqref{eqsxa}. 
Then the unique solution to \eqref{3eq12} exists globally. 
\end{lemma}

\begin{proof}
Assume the contrary, i.e, that $u_\delta(t)$ exists on interval $[0,T)$ but does not
exist on $[T,T+d]$, where $d>0$ is arbitrary small. Let us prove that 
the following limit 
\begin{equation}
\label{eq9.1}
\lim_{t\to T} u_\delta(t) = u_\delta(T)
\end{equation} 
exists and is finite. This contradicts the definition of $T$ since 
one can consider $u_\delta(T)$ as an initial data and construct the solution $u_\delta(t)$ 
on interval $[T,T+d]$, for sufficiently small $d>0$, by using the local existence of $u_\delta(t)$. 

Let us first prove that
\begin{equation}
\label{eqoo1}
\|u_\delta(t) - V_\delta(t)\| \le e^{-\phi(t)}\|w(0)\| + 
e^{-\phi(t)}\int_0^t e^{\phi(s)}\frac{|\dot{a}(s)|}{a(s)}\|V_\delta(s)\|ds, 
\end{equation}
for all $t\in [0,T)$ where
\begin{equation}
\label{eqoo2}
\phi(t):=\int_0^t a(s)ds,\qquad w(t) := u_\delta(t) - V_\delta(t).
\end{equation}

From \eqref{3eq12}, \eqref{2eq2}, and the monotonicity 
of $F$ one gets 
\begin{equation}
\label{eq>3}
\begin{split}
\langle \dot{u}_\delta(t), w(t) \rangle &= 
- \langle F(u_\delta(t)) + a(t)u_\delta(t) - F(V_\delta(t)) - a(t)V_\delta(t),w(t)\rangle \\
&\le - a(t) \|w(t)\|^2.
\end{split}
\end{equation}
It follows from \eqref{eq>2} and \eqref{eq>3} that
\begin{equation}
\label{eqoo3}
\begin{split}
\limsup_{\xi\to 0}&\frac{\|w(t+\xi )\|^2 - \|w(t)\|^2}{\xi } 
= \limsup_{\xi\to 0}\frac{\langle w(t+\xi ) - w(t), 
 w(t+\xi ) + w(t)\rangle }{\xi }\\
 \le &  2\langle \dot{u}_\delta(t), w(t)\rangle + 
\limsup_{\xi\to 0}\frac{\langle V_\delta(t+\xi ) - V_\delta(t), 
 w(t+\xi ) + w(t)\rangle }{\xi }\\
 \le & -2 a(t)\|w(t)\|^2 + 2\|w(t)\| \frac{|\dot{a}(t)|}{a(t)}\|V_\delta(t)\|.
\end{split}
\end{equation}
This and Lemma~\ref{lemma2.7} imply \eqref{eqoo1}. 

Let 
\begin{equation}
\label{eqoo6}
K = 1+\sup_{t\ge 0} e^{-\phi(t)} \int_0^t e^{\phi(s)} \frac{|\dot{a}(s)|}{a(s)}ds. 
\end{equation}
It follows from \eqref{eql5} that $K$ is bounded.
From \eqref{eqoo1}, \eqref{eqoo6}, and the fact that the function $\|V_\delta(t)\|$ is increasing, one obtains
\begin{equation}
\label{eq9.2}
\begin{split}
\|u_\delta(t)\| &\le e^{-\phi(t)}\|w(0)\| + K \|V_\delta(t)\|\\
&\le K_T:= \|w(0)\| + K \|V_\delta(T)\|,\qquad \forall t\in [0,T).
\end{split}
\end{equation}

Let $z_h(t):=u_\delta(t+h) - u_\delta(t)$. It follows from \eqref{3eq12} that
\begin{equation}
\label{eq9.3}
\begin{split}
\dot{z}_h(t) =& - [F(u_\delta(t+h)) - F(u_\delta(t)) + a(t)z_h(t)] \\
&+ (a(t)-a(t+h))u_\delta(t+h),\qquad 0<t<t+h<T. 
\end{split}
\end{equation}
Multiply \eqref{eq9.3} by $z_h(t)$ and use the monotonicity of 
$F$ to get
\begin{equation}
\label{eq9.4}
\|z_h(t)\|\frac{d}{dt}\|z_h(t)\| \le -a(t)\|z_h(t)\| ^2 + 
(a(t)-a(t+h))\langle u_\delta(t+h), z_h(t) \rangle.
\end{equation}
This and \eqref{eq9.2} imply
\begin{equation}
\label{eq9.10}
\begin{split}
\frac{d}{dt}\|z_h(t)\| &\le -a(t)\|z_h(t)\| + 
(a(t)-a(t+h))\|u_\delta(t+h)\|\\
&\le - a(t)\|z_h(t)\|+ 
(a(t)-a(t+h))K_T,\qquad 0<t<t+h<T.
\end{split}
\end{equation}
From \eqref{eq9.10} and the Gronwall's inequality one obtains
\begin{equation}
\label{eq9.11}
\begin{split}
\|z_h(t)\| &\le e^{-\phi(t)} \|z_h(0)\| + e^{-\phi(t)} K_T \int_0^t e^{\phi(s)}(a(s)-a(s+h))ds,\\
&\le \|z_h(0)\| + \max_{0\le s\le T-h}(a(s)-a(s+h)) K_T e^{-\phi(t)}\int_0^t e^{\phi(s)}ds,
\end{split}
\end{equation}
for $0<t<t+h<T$. 
It follows from \eqref{eq9.11} and the uniform continuity of $a(t)$ on $[0,T]$ that 
\begin{equation}
\label{eq9.12}
\lim_{h\to 0} \|u_\delta(t+h) - u_\delta(t)\| \le \lim_{h\to 0} \|u_\delta(0+h) - u_\delta(0)\| = 0,
\end{equation}
and this relation holds uniformly with respect to $t$ and $t+h$ such that 
$t<t+h<T$. 
Here, the last equality in \eqref{eq9.12} 
follows from the fact that $u_\delta(t)$ solves \eqref{3eq12} on $[0,T)$. 
Relation \eqref{eq9.12} and the Cauchy criterion for convergence imply 
the existence of the finite limit in \eqref{eq9.1}. 

Lemma~\ref{lemma9.4} is proved.
\end{proof}

\begin{theorem}
\label{theorem1}
Let $a(t)$ 
satisfy \eqref{eqsxa}. 
Assume that $F:H\to H$ is a monotone operator satisfying condition \eqref{eqholder},
and $u_0$ is an element of $H$, 
satisfying inequality 
\begin{equation}
\label{eqj1xx}
 \|F(u_0)-f_\delta\|>C\delta^\zeta>\delta,
\end{equation}
where $C>0$ and $0<\zeta\le 1$ are constants. 
 Assume that equation $F(u)=f$ has a solution, $f$ is unknown but $f_\delta$ is given, $\|f_\delta-f\|\le \delta$.
 Let $y$ be the minimal-norm solution to \eqref{aeq1}. 
Then the solution $u_\delta(t)$ to problem \eqref{3eq12}
exists globally and 
there exists a unique $t_\delta$ such that 
\begin{equation}
\label{2eq3xx}
\|F(u_\delta(t_\delta))-f_\delta\|=C\delta^\zeta,
\quad \|F(u_\delta(t))-f_\delta\|>C\delta^\zeta,\qquad \forall t\in[0,t_\delta). 
\end{equation}

If $\zeta\in (0,1)$ and 
\begin{equation}
\label{eqok1}
\lim_{\delta\to 0} t_\delta = \infty,
\end{equation}
then
\begin{equation}
\label{eqok2}
\lim_{\delta\to 0} \|u_\delta(t_\delta) - y\| = 0.
\end{equation}
\end{theorem}

\begin{remark}
{\rm
Inequality \eqref{eqj1xx} is not a restrictive assumption. Indeed, if it does not hold 
and $\|u_0\|$ is not too large, then $u_0$ can be considered as an approximate solution to \eqref{aeq1}. 

}
\end{remark}

\begin{proof}
The uniqueness of $t_\delta$ follows from \eqref{2eq3xx}. 
Indeed, if $t_\delta$ and $\tau_\delta>t_\delta$ both satisfy \eqref{2eq3xx}, then 
the second inequality in \eqref{2eq3xx} does not hold on the interval $[0,\tau_\delta)$.

{\it Let us verify the existence of $t_\delta$.} 

Denote 
\begin{equation}
\label{xex1}
v:=F(u_\delta)+a u_\delta - f_\delta,\quad h=\|v\|.
\end{equation}
We have
\begin{equation}
\label{eqoo8}
\begin{split}
\limsup_{\xi\to 0}&\frac{h^2(t+\xi ) - h^2(t)}{\xi } 
= \limsup_{\xi\to 0}\frac{\langle v(t+\xi ) - v(t), 
 v(t+\xi ) + v(t)\rangle }{\xi }\\
 \le &\limsup_{\xi\to 0}\frac{\langle F(u_\delta(t+\xi )) - F(u_\delta(t)), 
 v(t+\xi ) + v(t)\rangle }{\xi }\\
 &+   2\langle a\dot{u}_\delta(t), v(t)\rangle + 2\langle \dot{a}(t)u_\delta(t),v(t)\rangle .
 \end{split}
\end{equation}
From \eqref{3eq12} and \eqref{xex1} one gets $u_\delta(t+\xi) - u_\delta(t) = -\int_t^{t+\xi}v(s)ds$. 
This and the monotonicity of $F$ imply 
\begin{equation}
\label{eqoo9}
\bigg{\langle} F(u_\delta(t+\xi )) - F(u_\delta(t)), \int_t^{t+\xi} v(s)ds \bigg{\rangle} \le 0.
\end{equation}
Since $F$ is H\"{o}lder continuous of order $\alpha$ and $u_\delta(t)$ is differentiable one obtains
\begin{equation}
\label{eqoo10}
\big{\|}F(u_\delta(t+\xi)) - F(u_\delta(t))\big{\|} = O(|\xi|^\alpha),
\end{equation}
and
\begin{equation}
\label{eqoo11}
\bigg{\|}2\int_t^{t+\xi} v(s)ds - \xi\big{[}v(t+\xi) + v(t)\big{]} \bigg{\|}= O(|\xi|^{1+\alpha}).
\end{equation}
Relations \eqref{eqoo10}, \eqref{eqoo11} and the inequality $\alpha>1/2$ imply
\begin{equation}
\label{eqoo12}
\lim_{\xi\to 0} \frac{\langle F(u_\delta(t+\xi )) - F(u_\delta(t)), 
 v(t+\xi ) + v(t)  - \frac{2}{\xi}\int_t^{t+\xi}v(s)ds\rangle }{\xi } = 0.
\end{equation}
From \eqref{eqoo9} and \eqref{eqoo12} we get 
\begin{equation}
\label{eqoo13}
\limsup_{\xi\to 0}\frac{\langle F(u_\delta(t+\xi )) - F(u_\delta(t)), 
 v(t+\xi ) + v(t)\rangle }{\xi } \le 0.
\end{equation}
This, the relation $\dot{u}_\delta = -v$ (see \eqref{3eq12}), and \eqref{eqoo8} imply
\begin{equation}
\label{eqoo14}
\limsup_{\xi\to 0} \frac{h^2(t+\xi ) - h^2(t)}{\xi } 
\le -2 a(t)h^2(t) + 2|\dot{a}(t)|\|u_\delta(t)\|h(t). 
\end{equation}
This, Lemma~\ref{lemma2.7}, and \eqref{eq9.2} imply
\begin{equation}
\label{eqoo15}
h(t)\le e^{-\phi(t)}h(0) +  e^{-\phi(t)}\int_0^t e^{\phi(s)} |\dot{a}(s)|\bigg{(}\|w(0)\|) + K\|V_\delta(s)\|\bigg{)}ds.
\end{equation}

Since $\langle F(u_\delta)-F(V_\delta),u_\delta-V_\delta\rangle \ge 0$, one obtains two inequalities
\begin{equation}
\label{beq35}
a\|u_\delta - V_\delta\|^2 \le \langle v, u_\delta-V_\delta \rangle \le 
\|u_\delta - V_\delta\|h,
\end{equation}
and
\begin{equation}
\label{beq36}
\|F(u_\delta)-F(V_\delta)\|^2\le \langle v, F(u_\delta)-F(V_\delta) \rangle
\le h\|F(u_\delta)-F(V_\delta)\|.
\end{equation}
Inequalities \eqref{beq35} and \eqref{beq36} imply: 
\begin{equation}
\label{1eq16}
a\|u_\delta-V_\delta\|\le h,\quad \|F(u_\delta)-F(V_\delta)\|\le h.
\end{equation} 

The triangle inequality, the second inequality in \eqref{1eq16}, 
and \eqref{eqoo15} imply
\begin{equation}
\label{eql1}
\begin{split}
\|F(u_\delta(t)) - f_\delta\| \le
& \|F(V_\delta(t))-f_\delta\| + \|F(u_\delta)-F(V_\delta)\|\\
\le
& \|F(V_\delta(t))-f_\delta\| + h\\
\le & \|F(V_\delta(t))-f_\delta\| + 
h(0)e^{-\phi(t)}\\
& + e^{-\phi(t)}\int_0^t 
e^{\phi(s)}|\dot{a}(s)|\bigg{(}\|w(0)\|+ K\|V_\delta(s)\|\bigg{)}ds.
\end{split}
\end{equation}
This, \eqref{eql3}, \eqref{eqsxaz}, Lemma~\ref{lem1}, and \eqref{eqwer} imply
\begin{equation}
\label{eq**0}
\lim_{t\to\infty} \|F(u_\delta(t)) - f_\delta\| \le \lim_{t\to\infty} \|F(V_\delta(t))-f_\delta\| \le \delta.
\end{equation}
The existence of $t_\delta$ satisfying \eqref{2eq3xx} follows from \eqref{eq**0} 
and the continuity of the function $\|F(u_\delta(t)) - f_\delta\|$. 


{\it Let us prove \eqref{eqok2} given that \eqref{eqok1} holds.}

From \eqref{eqok1}, \eqref{eql6} with $r=1$, and the 
inequality $\|V_\delta(t)\|\ge \|V_\delta(0)\|>0$, $t\ge 0$, one gets, for all sufficiently small $\delta>0$, the following inequality:
\begin{equation}
\label{eqvv2}
h(0)e^{-\phi(t_\delta)}\le a(t_\delta)\|V_\delta(0)\| \le a(t_\delta)\|V_\delta(t_\delta)\|. 
\end{equation}
From the fact that $\|V_\delta(t)\|$ is a nondecreasing function of $t$, 
\eqref{eqwer2}, and \eqref{eqok1} 
one obtains
\begin{equation}
\label{eqvv1}
\begin{split}
Ke^{-\phi(t_\delta)}\int_0^{t_\delta} e^{\phi(s)}|\dot{a}(s)|
\|V_\delta(s)\| ds
&\le K\|V_\delta(t_\delta)\| e^{-\phi(t_\delta)}\int_0^{t_\delta} e^{\phi(s)}|\dot{a}(s)| ds\\
&\le a(t_\delta)\|V_\delta(t_\delta)\|,
\end{split}
\end{equation}
for all sufficiently small $\delta>0$. 
From \eqref{eqwer2} and \eqref{eqok1}
one gets, for all sufficiently small $\delta>0$, the following inequality
\begin{equation}
\label{eqvv1'}
\|w(0)\|e^{-\phi(t_\delta)}\int_0^{t_\delta} e^{\phi(s)}|\dot{a}(s)| ds
\le a(t_\delta)\|V_\delta(0)\|
\le a(t_\delta)\|V_\delta(t_\delta)\|.
\end{equation}


From \eqref{eqvv2}--\eqref{eqvv1'}, \eqref{2eq1}, and \eqref{eql1} with $t=t_\delta$, one obtains
\begin{equation}
\label{eql10}
C \delta^\zeta \le 4a(t_\delta)\|V_\delta(t_\delta)\| 
\le 4\bigg{(}a(t_\delta)\|y\| + \delta\bigg{)}.
\end{equation}
This and the relation $\lim_{\delta\to0}\frac{\delta^\zeta}{\delta} =\infty$
for a fixed $\zeta\in(0,1)$ imply
\begin{equation}
\label{eql11}
\lim_{\delta\to 0}\frac{\delta^\zeta}{a(t_\delta)} \le \frac{4\|y\|}{C}.
\end{equation}
Relation \eqref{eql11} and the first inequality in \eqref{2eq1} imply, for sufficiently small $\delta>0$,
the following inequality
\begin{equation}
\label{eql12}
\|V_\delta(t)\| \le \|y\|+\frac{\delta}{a(t_\delta)} < \|y\|+\frac{C\delta^\zeta}{a(t_\delta)}<
5\|y\|,\qquad 
0\le t\le t_\delta. 
\end{equation}
This implies
\begin{equation}
\label{eql13}
\lim_{\delta\to 0} \frac{\int_0^{t_\delta} e^{\phi(s)}|\dot{a}(s)|\|V_\delta(s)\|ds}
{e^{\phi(t_\delta)}a(t_\delta)} \le
5\|y\| 
\lim_{\delta\to 0}
\frac{\int_0^{t_\delta} e^{\phi(s)}|\dot{a}(s)|ds}
{e^{\phi(t_\delta)}a(t_\delta)}.
\end{equation}
It follows from \eqref{eql9} and \eqref{eql13} that
\begin{equation}
\label{eql15}
\lim_{\delta\to 0} \frac{\int_0^{t_\delta} e^{\phi(s)}|\dot{a}(s)|\|V_\delta(s)\|ds}
{e^{\phi(t_\delta)}a(t_\delta)} = 0.
\end{equation}

It follows from \eqref{1eq16} and \eqref{eqoo15} that 
$$
\|u_\delta(t) - V_\delta(t)\| \le h(0)\frac{e^{-\phi(t)}}{a(t)} +  
\frac{e^{-\phi(t)}}{a(t)}\int_0^t e^{\phi(s)} |\dot{a}(s)|\bigg{(}\|w(0)\| + K\|V_\delta(s)\|\bigg{)}ds.
$$
This, \eqref{eql9}, \eqref{eql6} with $r=1$, and \eqref{eql15} imply that
\begin{equation}
\label{eqj6}
\lim_{\delta\to 0} \|u_\delta(t_\delta) - V_\delta(t_\delta)\| = 0.
\end{equation}
From \eqref{eql11} one gets
\begin{equation}
\label{eqoo30}
\lim_{\delta\to 0}\frac{\delta}{a(t_\delta)} = 0.
\end{equation}

Now let us finish the proof of Theorem~\ref{theorem1}.

From the triangle inequality and inequality 
\eqref{rejected11} one obtains:
\begin{equation}
\label{eqhic56}
\begin{split}
\|u_\delta(t_\delta) - y\| &\le \|u_{\delta}(t_\delta) - V_{\delta}(t_\delta)\| + 
\|V(t_\delta) - V_\delta(t_\delta)\| + \|V(t_\delta) - y\|\\
&\le \|u_{\delta}(t_\delta) - V_{\delta}(t_\delta)\| + \frac{\delta}{a(t_\delta)} + \|V(t_\delta)-y\|.
\end{split}
\end{equation}
From \eqref{eqj6}--\eqref{eqhic56}, \eqref{eqok1}, and Lemma~\ref{rejectedlem}, one obtains
\eqref{eqok2}. Theorem~\ref{theorem1} is proved. 
\end{proof}

Assume that $a(t)$ satisfies the following conditions
\begin{equation}
\label{eqvv3}
0<a(t)\searrow 0,\qquad \frac{1}{3}>q > \frac{|\dot{a}(t)|}{a^2(t)}\searrow 0.
\end{equation}

\begin{remark}
{\rm
Let $a(t)=\frac{d}{(c+t)^b}$, 
where $b\in(0,1)$,\, 
$c>0$ and $d>bq^{-1}c^{b-1}$. Then this $a(t)$ satisfies \eqref{eqvv3}. 
}
\end{remark}

\begin{theorem}
\label{theorem2}
Let $a(t)$ 
satisfy \eqref{eqvv3}. 
Let $F$, $f$, $f_\delta$ be as in Theorem~\ref{theorem1}. 
Assume that $u_0\in H$  
satisfies either 
\begin{equation}
\label{eqj1}
\|F(u_0) + a(0)u_0 -f_\delta\| \le pa(0)\|V_\delta(0)\|,\qquad 0<p<1-\frac{q}{1-2q},
\end{equation}
or 
\begin{equation}
\label{eqj1xxx}
\|F(u_0) + a(0)u_0 -f_\delta\| \le \theta \delta^\zeta,\qquad 0\le \theta< C, 
\end{equation}
where $C>0$ is the constant from Theorem~\ref{theorem1}. 
Let $t_\delta$ be defined by \eqref{2eq3xx}. 
Then
\begin{equation}
\label{2eq4}
\lim_{\delta\to 0} t_\delta =\infty.
\end{equation}
\end{theorem}

\begin{remark}{\rm
One can easily choose $u_0$ satisfying inequality \eqref{eqj1}.
Indeed, \eqref{eqj1} holds if $u_0$ is sufficiently close to $V_\delta(0)$.
Note that inequality \eqref{eqj1} is a sufficient condition for 
\eqref{1eq20}, i.e.,
\begin{equation}
\label{eqj2}
e^{-\varphi(t)}h(0)\le pa(t)\|V_\delta(t)\|,\qquad t\ge 0, 
\end{equation}
to hold. In our proof inequality \eqref{eqj2} (or \eqref{1eq20})
is used at $t=t_\delta$.
The stopping time $t_\delta$ is often sufficiently large for 
the quantity $e^{\varphi(t_\delta)}a(t_\delta)$ to be large. 
 In this case  
inequality \eqref{eqj2} with $t=t_\delta$ is satisfied for a wide range of 
$u_0$. Note that by \eqref{eql6} one gets $\lim_{t\to\infty}e^{\varphi(t)}a(t)=\infty$. Here $\varphi(t) = (1-q)\phi(t)$ (see also \eqref{eqoo2} and \eqref{eq**}). 
}
\end{remark}

\begin{proof}[Proof of Theorem~\ref{theorem2}]
{\it Let us prove \eqref{2eq4} assuming that \eqref{eqj1} holds.} 
The proof goes similarly when \eqref{eqj1xxx} holds instead of \eqref{eqj1}. 

%
%

From \eqref{eqoo14} and the triangle inequality one gets
\begin{equation}
\label{eqoo16}
\begin{split}
\limsup_{\xi\to 0} \frac{h^2(t+\xi ) - h^2(t)}{\xi } 
\le& -2 a(t)h^2(t) + 
 2|\dot{a}(t)|\|V_\delta(t)\|h(t)\\
 &+ 2|\dot{a}(t)|\|u_\delta(t) - V_\delta(t)\|h(t)
\end{split}
\end{equation}
This and the first inequality in \eqref{1eq16} imply
\begin{equation}
\label{eqoo17}
\limsup_{\xi\to 0} \frac{h^2(t+\xi ) - h^2(t)}{\xi } 
\le -2 \bigg{(}a(t) - \frac{|\dot{a}(t)|}{a(t)}\bigg{)}h^2(t) + 
 2|\dot{a}(t)|\|V_\delta(t)\|h(t)
\end{equation}
%
Since $a-\frac{|\dot{a}|}{a}\ge (1-q)a$,  
by \eqref{eqvv3}, 
it follows from \eqref{eqoo17} and Lemma~\ref{lemma2.7} that
\begin{equation}
\label{1eq17}
h(t)\le h(0)e^{-\varphi(t)} + e^{-\varphi(t)}
\int_0^t e^{\varphi(s)}|\dot{a}(s)|\|V_\delta(s)\|ds,
\end{equation}
where
\begin{equation}
\label{eq**}
\varphi(t):=\int_0^t(1-q)a(s)ds = (1-q)\phi(t),\qquad t>0. 
\end{equation}
From \eqref{1eq17} and 
\eqref{1eq16}, one gets
\begin{equation}
\label{eqj5}
\|F(u_\delta(t))-F(V_\delta(t))\| \le 
h(0)e^{-\varphi(t)} + e^{-\varphi(t)}\int_0^t 
e^{\varphi(s)}|\dot{a}(s)|\|V_\delta(s)\|ds.
\end{equation}

It follows from inequality \eqref{eqj5} and the triangle inequality that
\begin{equation}
\label{1eq18}
\begin{split}
\|F(u_\delta(t))-f_\delta\|&\ge \|F(V_\delta(t))-f_\delta\|-\|F(V_\delta(t))-F(u_\delta(t))\|\\
&\ge a(t)\|V_\delta(t)\| - h(0)e^{-\varphi(t)} - 
e^{-\varphi(t)}\int_0^t e^{\varphi(s)} |\dot{a}| \|V_\delta\|ds.
\end{split}
\end{equation}
Since $a(t)$ satisfies \eqref{eqvv3} one gets by Lemma~\ref{lemauxi2} 
the following inequality
\begin{equation}
\label{1eq19}
\frac{q}{1-2q}a(t)\|V_\delta(t)\| \ge  e^{-\varphi(t)}\int_0^te^{\varphi(s)}|\dot{a}| \|V_\delta(s)\|ds.
\end{equation}

From the relation $h(t)=\|F(u_\delta(t))+a(t)u_\delta(t) -f_\delta\|$ (cf. \eqref{xex1}) 
and inequality \eqref{eqj1} one gets
\begin{equation}
\label{23eq6}
h(0)e^{-\varphi(t)}\le pa(0)\|V_\delta(0)\|e^{-\varphi(t)},\qquad 
t\ge 0.
\end{equation}
It follows from \eqref{eqsxa} that 
\begin{equation}
\label{26eq5}
e^{-\varphi(t)}a(0)\le a(t).
\end{equation}
Indeed, inequality $a(0)\le a(t)e^{\varphi(t)}$ is obviously true for 
$t=0$, and 
$$
\big(a(t)e^{\varphi(t)}\big)'_t = a^2(t)e^{\varphi(t)}\bigg{(}1-q - \frac{|\dot{a}(t)|}{a^2(t)}\bigg{)}\geq 0,
$$
by \eqref{eqsxa}. 
Here, we have used the relation $\dot{a} = -|\dot{a}|$ and 
the inequality $1-q>q$. 

Inequalities \eqref{23eq6} and \eqref{26eq5} imply
\begin{equation}
\label{1eq20}
e^{-\varphi(t)}h(0) 
\le p a(t)\|V_\delta(0)\|
\le p a(t)\|V_\delta(t)\|,\quad t\ge 0,
\end{equation}
where we have used the inequality $\|V_\delta(t)\|\le \|V_\delta(t')\|$ for $t\le t'$, 
established in Lemma~\ref{remark1}.
From \eqref{2eq3xx} and \eqref{1eq18}, \eqref{1eq19} and \eqref{1eq20}, one gets
$$
C\delta^\zeta = \|F(u_\delta(t_\delta))-f_\delta\|\ge (1 - p - \frac{q}{1-2q})a(t_\delta)\|V_\delta(t_\delta)\|.
$$
Thus,
\begin{equation}
\label{xex3}
\lim_{\delta\to0}a(t_\delta)\|V_\delta(t_\delta)\| = 0.
\end{equation}
From \eqref{rejected11} and the triangle inequality we obtain
\begin{equation}
\label{xex2}
a(t_\delta)\|V(t_\delta)\| \le a(t_\delta)\|V_\delta(t_\delta)\| + a(t_\delta)\|V(t_\delta)-V_\delta(t_\delta)\|
\le  a(t_\delta)\|V_\delta(t_\delta)\| +\delta.
\end{equation}
This and \eqref{xex3} imply
\begin{equation}
\label{exe4}
\lim_{\delta\to0}a(t_\delta)\|V(t_\delta)\| = 0.
\end{equation}
Since $\|V(t)\|$ is increasing and $\|V(0)\|>0$, relation \eqref{exe4} implies 
$\lim_{\delta\to0}a(t_\delta)=0$. 
Since $0<a(t)\searrow 0$, it follows that  
\eqref{2eq4} holds. 

Theorem~\ref{theorem2} is proved.
\end{proof}

If $F$ is a monotone operator then $F_1(u)=F(u + \bar{u})$, where $\bar{u}\in H$, is also
a monotone operator. Consider the following Cauchy problem
\begin{equation}
\label{eqv5}
\dot{u} = - (F(u) + a(t)(u - \bar{u}) -f_\delta),\qquad n\ge 0.
\end{equation}
Applying Theorem~\ref{theorem1} and Theorem~\ref{theorem2} for $F_1$ one gets the following corollaries:

\begin{corollary}
\label{corollary1}
Let $\bar{u}\in H$ be arbitrary and $y^*$ be the solution to \eqref{aeq1} with minimal distance to $\bar{u}$. 
Let $a(t)$ 
satisfy \eqref{eqsxa}. 
Assume that $F:H\to H$ is a monotone operator satisfying condition \eqref{eqholder},
and $u_0$ is an element of $H$, 
satisfying the inequality:
\begin{equation}
\label{ehjk1}
 \|F(u_0)-f_\delta\|>C\delta^\zeta>\delta,
\end{equation}
where $C>0$ and $0<\zeta\le 1$ are constants. 

Then the solution $u_\delta(t)$ to problem \eqref{eqv5}
exists globally, 
and 
there exists a unique $t_\delta>0$ such that 
\begin{equation}
\label{ehjk2}
\|F(u_\delta(t_\delta))-f_\delta\|=C\delta^\zeta,
\quad \|F(u_\delta(t))-f_\delta\|>C\delta^\zeta,\qquad \forall t\in[0,t_\delta). 
\end{equation}

If $\zeta\in (0,1)$ and 
\begin{equation}
\label{ehjk3}
\lim_{\delta\to 0} t_\delta = \infty,
\end{equation}
then
\begin{equation}
\label{ehjk4}
\lim_{\delta\to 0} \|u_\delta(t_\delta) - y^*\| = 0.
\end{equation}
\end{corollary}

\begin{corollary}
\label{corollary2}
Let $a(t)$ satisfy \eqref{eqvv3}. 
Let $F$ and $f_\delta$ be as in Corollary~\ref{corollary1}. 
Assume that $u_0$ be an element of $H$ such that $\|F(u_0)-f_\delta\|>C\delta^\zeta>\delta$. 
Assume in addition that $u_0$ satisfies either 
\begin{equation}
\label{eqjv1}
\|F(u_0) + a(0)(u_0-\bar{u}) -f_\delta\| \le p\|V_\delta(0)\|,\qquad 0<p<1-\frac{q}{1-2q},
\end{equation}
or
\begin{equation}
\label{eqjjvv1}
\|F(u_0) + a(0)(u_0-\bar{u}) -f_\delta\| \le \theta \delta^\zeta,\qquad 0<\theta<C. 
\end{equation}
where $C>0$ and $0<\zeta\le 1$ are constants from Corollary~\ref{corollary2}. 
Let $t_\delta$ be defined by \eqref{ehjk2}. 
Then
\begin{equation}
\label{2eqv4}
\lim_{\delta\to 0} t_\delta = \infty.
\end{equation}
\end{corollary}


\end{document}